\documentclass[12pt,preprint]{amsart}

\usepackage{amssymb,amsfonts,amsthm,amsmath,amscd}

\usepackage{enumerate}
\usepackage{verbatim}

\theoremstyle{plain}
\newtheorem{theorem}{Theorem}
\newtheorem{lemma}{Lemma}

\theoremstyle{definition}

\def\Z{\mathbb Z}
\def\Q{\mathbb Q}
\def\ssum{{\textstyle \sum}}

\begin{document}

\title{Sums of Quadratic residues and nonresidues}

\author{Christian Aebi}
\author{Grant Cairns}

\address{Coll\`ege Calvin, Geneva, Switzerland 1211}
\email{christian.aebi@edu.ge.ch}
\address{Department of Mathematics, La Trobe University, Melbourne, Australia 3086}
\email{G.Cairns@latrobe.edu.au}

\maketitle

\begin{abstract}
It is well known that when a prime $p$ is congruent to 1 modulo 4, the sum of the quadratic residues equals the sum of the quadratic nonresidues. In this note we give analogous results for the case where $p$ is congruent to 3 modulo 4.
\end{abstract}

 
Let $p$ be a prime with $p \equiv 3 \pmod 4$. Identify $\Z_p$ with the set $\{0,1,\dots,p-1\}$, and let $\Z_p^l=\{i\in \Z_p: 1<i<(p-1)/2\}$ and  $\Z_p^u= \{i\in \Z_p: i>(p-1)/2\}$.
Let $Q$ denote the set of quadratic residues of $\Z_p$ and write $Q^l=Q\cap \Z_p^l$ and $Q^u=Q\cap \Z_p^u$.
Similarly, let $N$ denote the set of quadratic nonresidues of $\Z_p$ and write $N^l=N\cap \Z_p^l$ and $N^u=N\cap \Z_p^u$.
For a set $A\subset \Z_p$, let us write $\ssum A$ for the sum in $\Z$ of the elements of $A$. 

The following Theorem may be deduced from well known formulas for the class number of the quadratic field $\Q[\sqrt{-p}]$; see \cite[p.~202]{Ma} and \cite[Chap.~6 (19)]{Da}. The purpose of this note is to give an elementary direct proof.
 Part (a) is stated in \cite[Cor.~13.4]{Be}. We have not seen part (b) stated in the literature.

\begin{theorem} \ 
\begin{enumerate}
\item[(a)] If $p= 7 \pmod 8$, then $\ssum Q^l=\ssum N^l$.
\item[(b)] If $p= 3 \pmod 8$, then $\ssum Q + \ssum Q^l= \ssum N+\ssum N^l$.
\end{enumerate}
\end{theorem}

\begin{proof} Let $n$ denote the number of quadratic nonresidues that are less than $p/2$. 
 Note that $-1$ is a quadratic nonresidue, since $p\equiv3\pmod4$; see \cite[Chap.~24]{Sil}. So for each quadratic nonresidue $q$, the element $p-q$ is a quadratic residue. 
In particular, $Q^u$ has $n$ elements.

\begin{lemma}\
\begin{enumerate}[{\rm (a)}]
\item[(a)] If $p= 7 \pmod 8$, then $\ssum Q=n p$.
\item[(b)] If $p= 3 \pmod 8$, then $3 \ssum Q=n p+\begin{binom}p2\end{binom}$.
\end{enumerate}
\end{lemma}

\begin{proof}
Let $\sigma$ be the doubling function $x\mapsto 2x$ on $\Z$, and  consider the function $\bar\sigma$ induced on $\Z_p$ by $\sigma$. 
Notice that if $x\in Q^l$, then $\bar\sigma (x) =\sigma (x)=2x$, and  if $x\in Q^u$, then $\bar\sigma (x) =\sigma (x)-p=2x-p$. 

(a)  When $p \equiv 7\mod8$, the function $\bar\sigma$ preserves $Q$. Now, since $Q^u$ has $n$ elements, 
\[
\ssum Q=\ssum \bar\sigma (Q)=\ssum \sigma (Q^l)+ \ssum \sigma (Q^u)-np=\ssum \sigma (Q)-np,
\]
that is,
$\ssum Q=2\ssum Q-np$, giving $\ssum Q=np$, as required.

(b)  When $p \equiv 3\pmod8$, the function $\bar\sigma$ sends quadratic residues to quadratic nonresidues.  Now, since $Q^u$ has $n$ elements, and $\ssum \Z_p =\begin{binom}p2\end{binom}$,
\begin{align*}
\ssum Q=\begin{binom}p2\end{binom} -\ssum \bar\sigma (Q)&=\begin{binom}p2\end{binom}-\left(\ssum \sigma (Q^l)+ \ssum \sigma (Q^u)-np\right)\\
&=\begin{binom}p2\end{binom}-\ssum \sigma (Q)+np.
\end{align*}
Thus $\ssum Q=\begin{binom}p2\end{binom}-2\ssum Q+np$, giving $3\ssum Q=np+\begin{binom}p2\end{binom}$, as required.
\end{proof}

Turning to the proof  of the Theorem, 
 note that 
\begin{equation}\label{equ1}
n p=  np -\ssum N^l+\ssum N^l= \ssum Q^u+\ssum N^l= \ssum Q+\ssum N^l-\ssum Q^l.
\end{equation}
Thus, when $p \equiv 7\pmod8$,   since $np=\ssum Q$ by the Lemma, we have $\ssum N^l=\ssum Q^l$. 

Similarly, if $p \equiv 3\pmod8$, then by the Lemma and equation (\ref{equ1}) we obtain
\begin{align*}
3 \ssum Q=n p+\begin{binom}p2\end{binom} &\implies  3 \ssum Q = \ssum Q+\ssum N^l-\ssum Q^l +\begin{binom}p2\end{binom}\\
&\implies  2 \ssum Q = \ssum N^l-\ssum Q^l +\ssum Q + \ssum N\\
&\implies  \ssum Q -\ssum N= \ssum N^l-\ssum Q^l \\
&\implies  \ssum Q + \ssum Q^l= \ssum N+\ssum N^l.
\end{align*}
\end{proof}


\begin{thebibliography}{1}

\bibitem{Be}
Bruce~C. Berndt, \emph{Classical theorems on quadratic residues}, Enseignement
  Math. (2) \textbf{22} (1976), no.~3--4, 261--304.

\bibitem{Da}
Harold Davenport, \emph{Multiplicative number theory}, third ed., Graduate
  Texts in Mathematics, vol.~74, Springer-Verlag, New York, 2000.

\bibitem{Ma}
Daniel~A. Marcus, \emph{Number fields}, Springer-Verlag, 1995.

\bibitem{Sil}
Joseph~H. Silverman, \emph{A friendly introduction to number theory}, third
  ed., Prentice Hall, New York, 2006.

\end{thebibliography}
\end{document}